\documentclass{article}
\usepackage[utf8]{inputenc}

\title{On uniform consistency of nonparametric tests II}

\author{Mikhail Ermakov }
\date{October 2020}

\usepackage{srcltx}
\usepackage{graphicx,color,amssymb,latexsym,amsmath,amsfonts}
\usepackage{mathrsfs}

\RequirePackage{amsthm,amsmath,amssymb}
\RequirePackage[numbers]{natbib}

\numberwithin{equation}{section}
\theoremstyle{plain}
\newtheorem{theorem}{Theorem}[section]

\newtheorem{lemma}{Lemma}[section]
\newtheorem{remark}{Remark}[section]

\newtheorem{proposition}{Proposition}[section]

\begin{document}
\global\long\def\zb{\boldsymbol{z}}
\global\long\def\ub{\boldsymbol{u}}
\global\long\def\vb{\boldsymbol{v}}
\global\long\def\wb{\boldsymbol{w}}
\global\long\def\tb{\boldsymbol{t}}
\global\long\def\eb{\boldsymbol{e}}
\global\long\def\sb{\boldsymbol{s}}
\global\long\def\0b{\boldsymbol{0}}
\global\long\def\xb{\boldsymbol{x}}
\global\long\def\xib{\boldsymbol{\xi}}
\global\long\def\etab{\boldsymbol{\eta}}
\global\long\def\omegab{\boldsymbol{\omega}}
\maketitle
Institute of Problems of Mechanical Engineering RAS, Bolshoy pr., 61, VO, 1991178  St. Petersburg and
St. Petersburg State University, Universitetsky pr., 28, Petrodvoretz, 198504 St. Petersburg, RUSSIA


keywords :
 Kolmogorov test,
consistency,
goodness of fit tests

\footnote{Research has been supported RFFI Grant  20-01-00273.}

\begin{abstract} For Kolmogorov type tests we point out necessary and sufficient   conditions of uniform consistency of sets of alternatives approaching to hypothesis. Sets of alternatives can be defined both in terms of distribution functions and in terms of densities.
\end{abstract}
\section{Introduction}
Let $X_1,\ldots,X_n$ be sample of i.i.d.r.v's having c.d.f.   $F_n\in \Im$, where $\Im$ is set of all c.d.f.'s of random variables taking values into interval $[0,1]$.

Define c.d.f. $F_0(x)= x$ for $x \in [0,1]$.

  We verify hypothesis
\begin{equation}\label{i1}
\mathbb{H}_0 : F_n = F_0,
\end{equation}
versus sets of alternatives defined in terms of

distribution functions
\begin{equation}\label{ax2}
\mathbb{H}_n \,:\, F_n \in \Upsilon_n, \quad \Upsilon_n \subset \Im
\end{equation}
or densities $p_n(x) = 1 + f_n(x) = \frac{d\,F_n(x)}{d\,x},$
\begin{equation}\label{axg}
\mathbb{H}_{1n} \,:\, f_n \in \Psi_n, \quad \Psi_n \subset \mathbb{L}_\infty(0,1).
\end{equation}
Here $\mathbb{L}_\infty(0,1)$ is Banach space of functions $h(x)$, $x \in [0,1]$, with the norm $\| h \|_\infty = \mbox{ess}\sup_{x \in (0,1)} |h(x)|$.

Denote $\hat F_n$  empirical c.d.f. of sample $X_1,\ldots,X_n$.

On the set $\Im$  of  distribution functions  define functional
$$
T(F) = \max_{x \in [0,1]} |F(x) - F_0(x) |
$$
Then $T(\hat F_n)$ is {\sl Kolmogorov test statistics}.

We explore  conditions on nonparametric sets of alternatives  $\Upsilon_n$ ( $\Psi_n$ respectively), under which Kolmogorov test is uniformly consistent. We call such sequences of sets of alternatives uniformly consistent.

Consistency and relative efficiency of Kolmogorov test was explored mainly for parametric sets of alternatives \cite{chib,durb,ha,ken,wel}. If sets of alternatives are Besov bodies  in Besov space $\mathbb{B}^s_{2\infty}$ with deleted "small ball" in   $\mathbb{L}_2(0,1)$, uniform consistency of Kolmogorov test has been explored Ingster \cite{ing87}. The  definition of uniform consistency in \cite{ing87} is different and can be considered as a version of uniform $\alpha$-consistency of this paper.

For  nonparametric test statistics $T_n(\hat F_n)$ having quadratic structure \cite{er97, er19} we show that separation of their normalized distances
$e_n \inf\,\{T_n(F)\,:\, F \in \Upsilon_n\,\}$ from zero is necessary and sufficient condition of uniform consistency of sets of alternatives  $\Upsilon_n$. Here  functional $T_n$ can depend on sample size $n$. Normalizing constants $e_n \to \infty$ as $n \to \infty$ are defined in a special way.

This statement has been proved in \cite{er97, er19} for nonparametric test statistics of

Cramer- von Mises tests;

  chi-squared test with number of cells increasing with growing sample size;

 tests generated quadratic forms of estimators of Fourier coefficients of signal in problem of signal detection in Gaussian white noise;

   tests generated $\mathbb{L}_2$ --norms of kernel estimators.

 We could not prove similar statement for Kolmogorov test.  Kolmogorov test
is  consistent  \cite{le} for sequences of alternatives $\mathbb{H}_n : F = F_n$, if $n^{1/2} T(F_n) \to \infty$ as  $n \to \infty$, and inconsistent if $n^{1/2} T(F_n) \to 0$ as  $n \to \infty$. At the same time Kolmogorov test  is inconsistent (asymptotically unbiased) \cite{le, massey, mason} for sets of alternatives
$$
\Theta_n(0,1,a)= \{F : n^{1/2} T(F) > a, F \in \Im \}
$$
with $a > 0$.

In Theorem \ref{tdf2} we show uniform consistency of Kolmogorov test on sets of alternatives
$$
\Upsilon_n(e_1,e_2,a)= \{F : n^{1/2} \max_{e_1 \le x \le e_2} |F(x) - F_0(x)| > a,\,
  n^{1/2}(F - F_0)\in \Im_{1} \},
$$
for arbitrary choice of $e_1, e_2$, $ 0 < e_1 < e_2 <1$ and $a > 0$. Here  $\Im_{1}$ is arbitrary subset of  differentiable  functions $G$ such that set of  functions $g(x) = \frac{dG(x)}{d\,x}$  is compact in $\mathbb{L}_2([0,1])$.

 In Theorem \ref{tdf2} we provide condition of uniform consistency of Kolmogorov test for goodness of fit testing. From Proposition \ref{tdf21} and proof of Proposition  \ref{pro2} it is easy to see that similar statement of uniform consistency of Kolmogorov type tests having weight function can be provided for problems of goodness of fit testing with parametric composite hypothesis (see \cite{durb} and Ch 5.5 \cite{wel}), nonparametric hypothesis testing on regression \cite{ha} and nonparametric hypothesis testing on symmetry (see Ch 2.2, exercise 5, \cite{wel}).  Similar results hold also for problems of hypothesis testing on homogeneity.

After that we explore consistency problem for nonparametric sets of alternatives  defined in terms of densities   (\ref{axg}).
In our setup
 sequence of sets of alternatives  $\Psi_n$ is uniformly consistent, if and only if, all sequences of simple alternatives $f_n \in \Psi_n$ are consistent.
Thus uniform consistency of sets of alternatives defined in terms of densities can be explored in the framework of  consistency of sequences of simple alternatives.

To explore consistency problem for sequences of simple alternatives we answer on the following three questions  {\it ii. - iv.}, posed in \cite{er19}.

Fix $r$, $0 <r \le 1/2$.  Suppose sets of alternatives are the following
$$
V_n = \{  f \,:\, \|f\|_\infty > c n^{-r}, \, f \in U\},
$$
where $U$ is  center-symmetric convex set,   $U \subset \mathbb{L}_\infty(0,1)$.
\vskip 0.15cm
{\it ii.} { \sl How to set   largest set $U$ such that Kolmogorov test is uniformly consistent on sets of alternatives  $V_n$ ?}
\vskip 0.15cm
 Such largest sets $U$ we call {\sl maxisets}. Exact definition of maxisets is provided in subsection \ref{subs22}.

 For $0< r < 1/2$, we show  (see Theorem \ref{tp2}) that maxisets are  Besov bodies in Besov space $\mathbb{B}^s_{\infty,\infty}$ with $s = \frac{2r}{1-2r}$, $r = \frac{s}{2 + 2s}$. Asymptotically minimax tests for Besov bodies in Besov space $\mathbb{B}^s_{\infty,\infty}$ has been found in \cite{lep}.

 If $ r = 1/2 $, we could not find sets $U$ satisfying all the requirements  in definition of maxisets defined for $ 0 < r < 1/2$. However, we show that sets of functions having a  finite fixed number of nonzero Fourier coefficients whose  absolute values do not exceed a certain constant satisfy similar requirements. In further statements of this section, and therefore in the corresponding Theorems, the maxisets can be replaced with such sets if $r=1/2$.
\vskip 0.2cm
{\sl iii. How to describe all consistent and inconsistent sequences of simple alternatives having fixed rate of convergence to hypothesis in $\mathbb{L}_\infty$ -- norm.}
\vskip 0.2cm
   We explore this problem as a problem of testing simple hypothesis
   \begin{equation}\label{i29}
\mathbb{\bar H}_0\,:\, f(x)  = 0, \quad x\in [0,1]
\end{equation}
   versus sequence of simple alternatives
 \begin{equation}\label{i30}
\mathbb{\bar H}_n\,:\, f = f_n, \qquad cn^{-r} \le \| f_n\|_\infty \le C n^{-r}, \qquad  0 < r \le 1/2,
\end{equation}
where $0 < c < C <\infty$.

In Theorem \ref{tp3} we show that functions
     $f_n$, $cn^{-r} \le \| f_n\|_\infty \le C n^{-r}$,  of any consistent sequence of alternatives  admit representation as sums of functions   $f_{1n}$, $c_1n^{-r} \le \| f_{1n}\|_\infty \le C_1 n^{-r}$, belonging to some maxiset and  functions $f_n - f_{1n}$ having the same signs of Fourier coefficients as functions $f_{1n}$ in wavelet basis defined a special way. Moreover, for any  $\varepsilon > 0$ there are maxiset and functions   $f_{1n}$, $c_1n^{-r} \le \| f_{1n}\|_\infty \le C_1 n^{-r}$, from the maxiset such that difference of type II error probabilities for alternatives  $f_n$ and $f_{1n}$ does not exceed $\varepsilon$   and functions $f_n$, $f_{1n}$, $f_n - f_{1n}$  have the same signs in Fourier coefficients in this special wavelet basis.
\vskip 0.2cm
{\sl iv. What can we say about interconnection of consistent and inconsistent sequences of alternatives having fixed rate of convergence to hypothesis?}
\vskip 0.2cm
 We show that type II error probabilities of alternatives formed by sums of functions of  consistent and inconsistent sequences of simple alternatives have the same asymptotic as type  II error probabilities of consistent sequence (see Theorem \ref{tp4}). We find analytic characterization of consistent sequences of alternatives $f_n$, $c n^{-r} < \|f_n\|_\infty < C n^{-r}$, such that these sequences do not contain as additive components of inconsistent sequence of alternatives with the same rate of convergence to hypothesis. We explore properties of such sequences.
We call such consistent sequences of alternatives {\sl pure consistent sequences}. These properties for Kolmogorov test are akin to properties of nonparametric test statistics having quadratic structure \cite{er19}.

As mentioned, we show that maxisets for Kolmogorov test are Besov bodies in Besov space $\mathbb{B}^s_{\infty,\infty}$.

Define Besov bodies in the following form
$$
\mathbb{B}^s_{\infty,\infty}(P_0) = \Bigl\{f : f = \sum_{k=1}^\infty \sum_{j=1}^{2^k} \theta_{kj} \phi_{kj},\,\, \sup_k 2^{(s+1/2)k}\sup_{1 \le j \le 2^k} |\theta_{kj}| \le P_0,\, \theta_{kj} \in \mathbb{R}^1\Bigr\},
$$
where $\phi_{kj}(x) = \phi(2^k x - j)$, $x \in (0,1)$,  is wavelet basis and $P_0 > 0$. We suppose that mother periodic wavelet  $\phi$ has smoothness more than $[s] +2$ and has bounded support $(a_1,a_2)$, moreover
$\int_{a_1}^{a_2} x \phi(x + c)\,dx = 0$, with $c = (a_1 + a_2)/2$. Here $m = [s]$ denote a whole part of $s$.

Note  (see \cite{jo}) that, for non-integer $s$ with $m = [s]$, Besov space $\mathbb{B}^s_{\infty,\infty}$ is set of functions $f$ such that there exist $\frac{d^m f(x)}{d x^m}$ for all $x \in (0,1)$ with $m = [s]$ and $\frac{d^m f(x)}{d x^m}$ satisfies Hoelder conditions with order $s -m$.  If $s$ is whole, then functions $\frac{d^{s-1} f(x)}{d x^{s-1}}$ belong to Zigmund class.

Paper is organized as follows. In section \ref{def} main definitions are provided. In section \ref{condis}  we state Theorems about uniform consistency of sets of alternatives defined in terms of distribution functions. In section \ref{conden}; for $0 < r < 1/2$, Theorems about consistency of sequences of simple alternatives defined in terms of densities are provided. Similar results for $r =1/2$ are provided  in section \ref{sos1}. Proof of Theorems one can find in Appendix.

We use letters $c$ and $C$ as a generic notation for positive constants. For any two sequences of positive real numbers $a_n$ and $b_n$,  $a_n \asymp b_n$ implies $c < a_n/b_n < C$ for all $n$ and  $a_n = o(b_n)$, $a_n =  O(b_n)$ imply $a_n/b_n \to 0$ as $n \to \infty$ and $a_n < C b_n$ for all $n$ respectively. Denote $[a]$ a whole part of any real number $a$.
\section{Main definitions \label{def}}

\subsection{Consistency \label{subs21}}
We begin with  the problem of testing hypothesis on a distribution function (\ref{i1}) versus alternatives (\ref{ax2}).

  For Kolmogorov test $K_n = K_n(X_1,\ldots,X_n)$ denote $\alpha(K_n)$ its type I error probability and $\beta(K_n,F_n)$ its type II error probability for alternative $F_n \in \Im$.

For sets of alternatives $\Upsilon_n$, $\Upsilon_n \subset \Im$, denote
$$
\beta(K_n,\Upsilon_n) = \sup_{F \in \Upsilon_n} \beta(K_n,F).
$$
We say that sequence of sets of alternatives $\Upsilon_n$ is $\alpha$--uniformly consistent for fixed type I error probability  $\alpha$, $0 < \alpha < 1$, if, for Kolmogorov test $K_n$, $\alpha(K_n) =  \alpha + o(1)$ as $n \to \infty$, there holds
$$
\limsup_{n \to \infty} \beta(K_n,\Upsilon_n) < 1 -  \alpha.
$$
In what follows, we  shall call such sequences of sets of alternatives $\Upsilon_n$  $\alpha$-uniformly consistent.

If sequence of sets of alternatives  $\Upsilon_n$ is $\alpha$-uniformly consistent  for all $\alpha$, $0 < \alpha <1$, we say that sequence of sets of alternatives  $\Upsilon_n$ is uniformly consistent.

For problem of testing hypothesis on  a density (\ref{i29}) versus alternatives (\ref{i30}) we also implement the notation  $\beta(K_n,f_n)$ for type II error probability of alternative  $f_n$.

We say that sequence of simple alternatives $f_n$ is $\alpha$-consistent, $0< \alpha < 1$, if for tests $K_n$, $\alpha(K_n) =  \alpha + o(1)$ as $n \to \infty$, we have
$$
\limsup_{n \to \infty} \beta(K_n,f_n) < 1 -  \alpha
$$
If sequence of simple alternatives  $f_n$ is $\alpha$-consistent for all $0< \alpha < 1$, then we say that this sequence is consistent. If sequence of simple alternatives $f_n$ is not $\alpha$-consistent for all $\alpha$, $0 < \alpha <1$, then we say that this sequence is inconsistent.

For sequences of simple alternatives $F_n \in \Im$ we  shall implement the same terminology.

\subsection{Definition of maxiset and maxispace \label{subs22}}

We explore problem of hypothesis  testing on density (\ref{i29}) versus simple alternatives (\ref{i30}).

Functions arising in  reasoning should be densities. To guarantee  this property we introduce the following condition.

\noindent{A}. For sequence of functions $f_n= \sum_{j=1}^\infty \sum_{i=1}^{2^j} \theta_{nji} \phi_{ji}$ there is $l_0$, such that for any $l > l_0$ functions
$$
1 + \sum_{k=1}^l \sum_{j=1}^{2^k} \theta_{nkj} \phi_{kj},
\quad  \quad
1 + \sum_{k=l}^\infty \sum_{j=1}^{2^k} \theta_{nkj} \phi_{kj}
$$
are densities.

If all $f_n = f$ for all $n$,  we say that function $f$  satisfies $A$.

Let $\Xi$, $\Xi \subset \mathbb{L}_\infty(0,1)$,  be Banach  space with  norm $\|\cdot\|_\Xi$. Denote $ U=\{f:\, \|f\|_\Xi \le r,\, f \in \Xi\}$ ball in $\Xi$ having radius  $r> 0$.

Definition of maxiset $U$  and maxispace  $\Xi$ is akin to \cite{er19}.  The definition is provided in terms of   $\mathbb{L}_\infty$-norm.

Define subspaces $\Pi_k$, $1 \le k < \infty$, by induction.

Put $d_1= \max\{\|f\|_\infty,\, f \in U\}$. Define function  $e_1 \in U$ such that  $\|e_1\|_\infty= d_1.$ Denote $\Pi_1$ linear subspace generated  $e_1$.

For $i=2,3,\ldots$, denote
$d_i = \max\{\rho(f,\Pi_{i-1}), f \in U \}$, where $\rho(f,\Pi_{i-1})=\min\{\|f-g\|_\infty, g \in \Pi_{i-1} \}$. Define function $e_i \in U$ such that $\rho(e_i,\Pi_{i-1}) = d_i$.
Denote $\Pi_i$ linear subspace generated functions $e_1,\ldots,e_i$.

Let $\tb_1,\tb_2,\ldots$ be orthonormal basis in $\mathbb{L}_2(0,1)$ and let functions $\tb_i$, $i=1,2,\ldots,$ be bounded. For each $i$, $i=1,2,\ldots,$ denote $\Gamma_i$-linear subspace generated functions $\tb_1,  \ldots, \tb_i$.

For each $i$, $i=1,2,\ldots,$ denote $a_i = \sup\{\rho(f,\Gamma_i)\,:\, f \in U\}$.

Suppose that $a_i \asymp d_i$ as $i \to \infty$.

For any bounded function $f$, denote $b_i(f) = \inf\,\{\|f - g\|\,|\, g \in \Gamma_i\}$.

For any function $f \in \mathbb{L}_\infty(0,1)$, define function
$f_{\Gamma_i} \in \Gamma_i$, $i=1,2,\ldots,$ such that $\| f - f_{\Gamma_i}\| = b_i(f)$. Put $\tilde f_i = f - f_{\Gamma_i}$.

Ball $U$ is called {\sl maxiset} for test statistics  $T$ of Kolmogorov test and functional space  $\Xi$ is  called {\sl maxispace} for basis $\tb_1,\tb_2,\ldots$, if following two conditions are satisfied:
\vskip 0.2cm
{\sl i.} for any subsequence of simple alternatives $f_{n_j} \in U$, $cn_j^{-r} < \|f_{n_j}\|_\infty < Cn_j^{-r}$, $n_j \to \infty$ as $j \to \infty$, there is $\alpha_0$, $0 < \alpha , \alpha_0$ such that this subsequence  is $\alpha$-- consistent  for subsequences of samples $X_1,\ldots,X_{n_j}$ respectively.
\vskip 0.2cm
{\sl ii.}  for any  $f \in \mathbb{L}_\infty(0,1)$, $f \notin \Xi$, and $f$ satisfies A, there are sequences $i_n$, $j_{i_n}$, $i_n \to \infty$, $j_{i_n} \to \infty$ as $n \to \infty$   such that $c j_{i_n}^{-r}<\|\tilde f_{i_n}\|_\infty < C j_{i_n}^{-r}$ and subsequence of alternatives $\tilde f_{i_n}$  is inconsistent for subsequence of samples $X_1, \ldots, X_{j_{i_n}}$.

\section{Consistency of  alternatives defined in terms of distribution functions \label{condis}}

\begin{theorem} \label{tdf1}
For each  $a >0$ there is $\alpha_0 = \alpha(a)$, $0 < \alpha_0 <1$, such that sets of alternatives $\Upsilon_n(0,1,a)$ are $\alpha$-uniformly consistent for $\alpha_0<\alpha <1$
\end{theorem}
\begin{theorem} \label{tdf4} Sequence of alternatives $F_n \in \Im$ is inconsistent if  there holds
\begin{equation*}
\lim_{n \to \infty} \sqrt{n}\, T(F_n) = 0.
\end{equation*}
\end{theorem}
Theorems \ref{tdf1} and  \ref{tdf4} follows straightforwardly  from  Dvoretzky-Kiefer-Wolfowitz inequality (see (14.7), Ch 14.2, \cite{le} and also \cite{dv}, \cite{gine}, \cite{mas}) and proof is omitted. Theorem \ref{tdf4} is  provided in \cite{le}, Ch. 14.2, Th. 14.2.3 and \cite{gine}.

 \begin{theorem} \label{tdf2} For   any $a > 0$ and for any $e_1$, $e_2$, $0< e_1 < e_2 <1$, sequence of sets of alternatives $\Upsilon_n(e_1,e_2,a)$ is uniformly consistent.
\end{theorem}
As has been shown by Massey  (see \cite{mas} and also Ch. 14, \cite{le})
there exist $\alpha_1$, $0 < \alpha_1 <1$, such that for Kolmogorov tests $K_n$, $\alpha(K_n) < \alpha_1$, $n > n_0(\alpha_0)$, there holds
\begin{equation*}
\limsup_{n \to \infty} \beta(K_n,F_n) = 1 -  \alpha.
\end{equation*}
Theorem \ref{tdf3} given below  is a version of Theorem 3 in  \cite{chib} (see also Th. 1 Ch.4.2 \cite{wel}). Proof of Theorem \ref{tdf3} is akin to proof of Theorem  1, Ch.4.2, \cite{wel} and is omitted.
\begin{theorem} \label{tdf3} Let $F_n(t)$, $t \in [0,1]$ be sequence of alternatives such that
$  \sqrt{n} T(F_n) < C$. Let functions $F_n$ be continuous and strictly increasing. Then there is probability space such that on this probability space, for independent identically distributed random variables
$X_1,\ldots,X_n$, having c.d.f.'s $F_n$ and Brownian bridges $b_n(t)$, $t \in [0,1]$,  there holds
\begin{equation*}
\lim_{n \to \infty} \mathbf{P}\,(\,\sup_{t \in [0,1]} |\,  \sqrt{n}\, (\hat F_n(t)\, - \,F_n(t) ) -\, b_n(t)\,| > \delta) =\, 0.
\end{equation*}
 for any $\delta > 0$.
\end{theorem}
\section{Consistency of  alternatives defined in terms of densities, $0< r < 1/2$ \label{conden}}

We explore problem of testing hypothesis on density (\ref{i29})   versus simple alternatives (\ref{i30}).

Denote $k_n = [(1/2 -  r) \log n]$

Introduce the following assumption.
\vskip 0.3cm
{\bf G.} For any  $\varepsilon > 0$ there is integer $c_\varepsilon$, such that, for all $n > n_0(\varepsilon, c_\varepsilon)$ and all $l < k_n - c_\varepsilon$ there holds
$$
n^{1/2} T(F_{nl}) < \varepsilon,
$$
where
$$
F_{nl}(x)= x + \sum_{j=1}^{l}\,\sum_{i=1}^{2^j} \theta_{nji} \psi_{ji}(x), \quad x \in [0,1].
$$
Here $\psi_{ji}(x) = \int_0^x \phi_{ji}(t) \, d\,t$, $x \in (0,1)$.

If $G$ does not valid, then consistency of Kolmogorov test holds for a faster rate of convergence of sequence of alternatives to hypothesis.

\begin{theorem} \label{tp1} Assume $G$.  There is $\alpha_0$, $0 < \alpha_0 < 1$, such that sequence of alternatives   $f_n= \sum_{j=1}^\infty \sum_{i=1}^{2^j} \theta_{nji} \phi_{ji}$, $\|f_n\|_\infty \asymp n^{-r}$ is $\alpha$-consistent for $\alpha_0< \alpha < 1$, if and only if, there is sequences $j_n$, $j_n - k_n = O(1)$ and $i_n$, $1 \le i_n \le 2^{j_n}$, such that there holds
$
 |\theta_{nj_ni_n}| > c n^{-1/4-r/2}$.

Sequence of alternatives $f_n$ is consistent, if and only if,  there are $e_1, e_2$, $0< e_1 < e_2 < 1$, such that $e_1 < i_n 2^{-j_n} < e_2$ for all $n  > n_0(e_1,e_2)$. Here $j_n - k_n = O(1)$ as $n \to \infty$.
\end{theorem}
Theorem \ref{tp101} can be considered as a Corollary of Theorem  \ref{tp1}.
\begin{theorem} \label{tp101} Assume $G$.  Sequence of alternatives $f_n= \sum_{j=1}^\infty \sum_{i=1}^{2^j} \theta_{nji} \phi_{ji}$, $\|f_n\|_\infty \asymp n^{-r}$, is inconsistent if there is  sequence $j_n$, $j_n - k_n \to \infty$, such that $ |\theta_{nj_ni_n}| = o(n^{-1/4-r/2})$ for all sequences $ l_n \le j_n$ and $1  \le i_n \le 2^{l_n}$.
\end{theorem}

\begin{theorem} \label{tp2} Besov bodies $\mathbb{B}^s_{\infty,\infty}(P_0)$  are maxisets for Kolmogorov test. Here $s = \frac{2r}{1 - 2r}$, $P_0 >0$.
\end{theorem}
We say that  functions $f_1 = \sum_{j=1}^\infty \sum_{i=1}^{2^j} \theta_{1ji} \phi_{ji}$ and $f_2 = \sum_{j=1}^\infty \sum_{i=1}^{2^j} \theta_{2ji} \phi_{ji}$  have the same orientation if $\theta_{1ji} \theta_{2ji} \ge 0$ for all $1 \le j < \infty$ and $1 \le i  \le 2^j$.

\begin{theorem} \label{tp2a} Assume G. Then sequence of alternatives is $\alpha$-consistent for some $0 < \alpha <1$,  if and only if, there are maxisets  $\mathbb{B}^s_{\infty,\infty}(P_0)$ and sequence of functions $f_{1n} \in \mathbb{B}^s_{\infty,\infty}(P_0)$, $\|f_{1n}\|_\infty \asymp n^{-r}$, such that functions $f_n$, $f_{1n}$ and $f_n - f_{1n}$ have the same orientation.
\end{theorem}
Let we have two sequences of functions $f_{1n} \in \mathbb{B}^s_{\infty,\infty}(P_0)$, $\|f_{1n}\|_\infty \asymp n^{-r}$ and $f_{2n}$, $\|f_{2n}\|_\infty \asymp n^{-r}$ such that $1 + f_{1n}$ and $1 + f_{2n}$ are densities, $f_{1n}$, $f_{2n}$ have the same orientation for all $n > n_0 >0$ and satisfy G. Then, by Theorem \ref{tp2a}, there is $\alpha$, $0 < \alpha < 1$ such that sequence of alternatives $(f_{1n} + f_{2n})/2$ is $\alpha$--consistent.
\begin{theorem} \label{tp3} Assume $A$ and $G$. Let sequence of alternatives  $f_n$, $c n^{-r} < \|f_n\|_\infty < C n^{-r}$, be $\alpha$ -consistent for some $0 < \alpha < 1$.  Let $A$ and $G$ hold. Then, for any $\varepsilon > 0$, there is maxiset $\mathbb{B}^s_{\infty,\infty}(P_0)$ and sequence $f_{1n} \in \mathbb{B}^s_{\infty,\infty}(P_0)$, $\|f_{1n}\|_\infty  \asymp n^{-r}$, such that

{\it i.} functions $f_n$, $f_{1n}$ and $f_n - f_{1n}$ have the same orientation,

 {\it ii.} for Kolmogorov  tests $K_n$, $\alpha(K_n) = \alpha + o(1)$, there is $n_0(\varepsilon)$ such that   there hold:
\begin{equation}\label{d1}
|\beta(K_n,f_n) - \beta(K_n,f_{1n})| < \varepsilon
\end{equation}
and
\begin{equation}\label{d2}
 \beta(K_n,f_n - f_{1n}) < \varepsilon
\end{equation}
for all  $n > n_0(\varepsilon)$.
\end{theorem}
\begin{theorem} \label{tp4} Let sequence of alternatives  $F_n$ be $\alpha$--consistent for some $0 < \alpha < 1$ and sequence of alternatives $F_{1n}$ be inconsistent. Let functions $F_n + F_{1n} - F_0$ be distribution functions. Let $F_n$ and $F_n + F_{1n} - F_0$  be strongly increasing and continuous. Then we have
\begin{equation}\label{d3}
\lim_{n \to \infty}|\beta(K_n,F_n) - \beta(K_n,F_n + F_{1n} - F_0)| = 0,
\end{equation}
where $\alpha(K_n) = \alpha + o(1)$, $0 < \alpha < 1$.
\end{theorem}
We say that $\alpha$- consistent sequence of alternatives  $f_n=\sum_{j=1}^\infty \sum_{i=1}^{2^j} \theta_{nji} \phi_{ji}$, $c n^{-r} < \|f_n\|_\infty < C n^{-r}$,  is {\sl purely $\alpha$--consistent}, if there does not exist inconsistent subsequence alternatives  $f_{1n_l}= \sum_{k=1}^\infty \sum_{j=1}^{2^k} \theta_{1n_lji} \phi_{kj}$, $\|f_{1n_l}\|_\infty \asymp n_l^{-r}$ such that, for all  $1 \le j < \infty$ and $1 \le i \le 2^j$,   there holds $(\theta_{n_lji} - \theta_{1n_lji})\theta_{n_lji} > 0$,  if $|\theta_{n_lji}| > 0$, and $\theta_{1n_lji} = 0$ if $\theta_{n_lji} =  0$.
\begin{theorem}\label{tq6}  Assume G. Sequence of alternatives $f_n$, $cn^{-r} \le \|f_n\|_\infty \le Cn^{-r}$, is purely $\alpha$-consistent for some $0 < \alpha < 1$, if and only if,  for any $\varepsilon >0$  there is  $C_1= C_1(\varepsilon)$, such that there holds
\begin{equation*}
\sup_{|j| > k_n + N_1} \sup_{1 \le i \le 2^j} 2^{j/2} n^r |\theta_{nji}| \le \varepsilon n^{-r}
\end{equation*}
for all $n> n_0(\varepsilon)$.
\end{theorem}
\begin{theorem}\label{tq12} Assume $A$ and $G$. Sequence of alternatives $f_n$, $c n^{-r}\le \|f_{n}\|_\infty \le C n^{-r}
$, is purely $\alpha$-consistent, if and only if,  for any $\varepsilon > 0$  there is maxiset $ \mathbb{ B}^s_{\infty\infty}(P_0)$ and sequence of functions  $f_{1n} \in \mathbb{ B}^s_{\infty\infty}(P_0)$   such that functions $f_n$, $f_{1n}$, $f_n - f_{1n}$ have the same orientation and $\|f_n - f_{1n}\|_\infty \le \varepsilon n^{-r}$.
\end{theorem}
One of the most popular definition of consistency \cite{le} is the following.  Sequence of alternatives $f_n$ is consistent if there is $\alpha_n>0$, $\alpha_n \to 0$ as $n \to \infty$, such that, for any sequence of tests $K_n$, $\alpha(K_n) \ge \alpha_n$, $\alpha(K_n) \to 0$ as $n\to \infty$  there holds $\beta(K_n,f_n) \to 0$ as $n \to \infty$. We slightly modify this definition  to implement to test statistics instead of tests.

Such a definition generates the following notion of uniform consistency. Sequence of sets of alternatives $\Psi_n$ is uniformly consistent if there is $\alpha_n>0$, $\alpha_n \to 0$ as $n \to \infty$, such that, for any sequence of tests $K_n$, $\alpha(K_n) \ge \alpha_n$, $\alpha(K_n) \to 0$ as $n\to \infty$  there holds $\beta(K_n,\Psi_n) \to 0$ as $n \to \infty$.

 For brevity, we shall talk about $s$-consistency and uniform $s$-consistency respectively in the case of these definitions.

From reasoning of the proof of Theorem \ref{tp1} it is easy to get the following Theorem \ref{tp112}.
\begin{theorem} \label{tp112} Subsequence of alternatives $f_{n_l}= \sum_{j=1}^\infty \sum_{i=1}^{2^j} \theta_{n_lji} \phi_{ji}$ is s-consistent, if and only if, there is sequences $j_{n_l}$, $j_{n_l} \to \infty$ as $l \to \infty$, and $i_{n_l}$, $1 \le i_{n_l}\le 2^{j_{n_l}}$, such that there holds
 \begin{equation*} n_l^{1/2} 2^{-j_{n_l}/2}|\theta_{n_lj_{n_l}i_{nl}}| \to \infty\quad \mbox{as}\quad l \to \infty.
 \end{equation*}
 \end{theorem}
 \begin{remark} Slightly modifying reasoning of proof of Theorem \ref{tp2} we can show following properties of $s$-consistent sequences of alternatives.

 Sequence of alternatives $f_n \in \mathbb{B}^s_{2\infty}$, $s =s = \frac{2r}{1 - 2r}$, is $s$-consistent if  and only if $n^{r} \,\|f_n\|_\infty \to \infty$ as $n \to \infty$.

 For any  $f \in \mathbb{L}_\infty(0,1)$, $f \notin  \mathbb{B}^s_{2\infty}$, and $f$ satisfies A, there are sequences $i_n$, $j_{i_n}$, $i_n \to \infty$, $j_{i_n} \to \infty$ as $n \to \infty$   such that $\|\tilde f_{i_n}\|_\infty  j_{i_n}^{-r}\to \infty$ as $n \to \infty$ and subsequence of alternatives $\tilde f_{i_n}$  is inconsistent for subsequence of samples $X_1, \ldots, X_{j_{i_n}}$.

Proof of these two statements is akin to the proof of Theorems \ref{tp1},  \ref{tp2a} respectively and is omitted.
 \end{remark}
 Similar Remark holds for setups of first part of paper \cite{er19}.
\section{Consistency of  alternatives defined in terms of densities, $ r = 1/2$ \label{sos1}}

We explore problem of hypothesis testing on density (\ref{i29}), (\ref{i30})
 with $r = \frac{1}{2}$.

\begin{theorem} \label{tco1} Sequence of alternatives $f_n= \sum_{j=1}^\infty \sum_{i=1}^{2^j} \theta_{nji} \phi_{ji}$, $\|f_n\|_\infty \asymp n^{-1/2}$, is $\alpha$-consistent, $\alpha_0< \alpha < 1$, if and only if, there are $C$, $c$ and sequences $j_n \le C $,  $i_n$, $1 \le i_n \le 2^{j_n}$, such that
$ n^{1/2}|\theta_{nj_ni_n}| > c$.

This sequence of alternatives $f_n$ is consistent, if and only if, we can point out additionally constants  $e_1, e_2$, $0< e_1 < e_2 < 1$, such that $e_1 < i_n 2^{-j_n} < e_2$ for all $n  > n_0(e_1,e_2)$.
\end{theorem}
 Thus problem of consistency for alternatives approaching with hypothesis with rate $n^{-1/2}$ is reduced to  finite parametric problem of testing hypothesis such that sets of alternatives of this parametric problem contains  most informative part of orthogonal expansions  of alternatives   $f_n$.

Proof of Theorem \ref{tco1} is akin to proof of Theorem \ref{tp1} and is omitted.

Theorem \ref{tco2} given below shows that, for this setup, the linear space
\begin{equation*}
\Xi = \Bigl\{f \,:\, f = \sum_{j=1}^m \sum_{i=1}^{2^j} \theta_{ji} \phi_{ji},\,  \theta_{ji} \in \mathbb{R}^1, m =1,2,\ldots\,\Bigr\}
\end{equation*}
satisfies {\it ii} in definition of maxisets in previous setup.

Sets
\begin{equation*}
U(m, P_0) = \,\Bigl\{f \,:\, f = \sum_{j=1}^m \sum_{i=1}^{2^j} \theta_{ji} \phi_{ji},\, \|f\|_\infty < P_0,\,  \theta_{ji} \in \mathbb{R}^1\Bigr\}.
\end{equation*}
satisfies {\it i.} in definition  of maxisets of previous setup.

Theorem \ref{tco2} given below is akin to Theorem \ref{tp2}.
\begin{theorem} \label{tco2}  If there are $m$ and $P_0 > 0$ such that $f_n \in U(m,P_0)$, $\|f_n\|_\infty \asymp n^{-1/2}$, then sequence of alternatives $f_n$ is consistent.

 Let $f = \sum_{j=1}^\infty \sum_{i=1}^{2^j} \theta_{ji} \phi_{ji} \notin \Xi$ and $f$ satisfy A.  Then there are sequence $m_n \to \infty$ and subsequence $l_{m_n} \to \infty$ as $n \to \infty$ such that, for functions   $f_{l_{m_n}}   =  \sum_{j=m_n}^\infty \sum_{i=1}^{2^j} \theta_{ji} \phi_{ji}$, there hold $\|f_{l_{m_n}}\|_\infty \asymp l_{m_n}^{-1/2}$ as $n \to \infty$ and subsequence of alternatives $f_{l_{m_n}}$ is inconsistent for subsequence of samples $X_1, \ldots, X_{l_{m_n}}$.
\end{theorem}
\begin{theorem} \label{tco3} Let $r = 1/2$ and assumption G is omitted. Then statements of Theorems \ref{tp2a} and  \ref{tp3} hold if maxisets  $\mathbb{B}^s_{\infty,\infty}(P_0)$ are replaced with sets $U(m.P_0)$.
\end{theorem}
Proof of Theorem \ref{tco3} is akin to proof of Theorems  \ref{tp2a} and \ref{tp3} and is omitted.
\section{Appendix}
\subsection{Proof of Theorem \ref{tdf2}}
Denote $b(t)$, $t \in [0,1]$, Brownian bridge, $\mathbf{E} b (t) = 0$ and $\mathbf{E} [b(t)\, b(s)] = min(t,s) - s t$, $t,s \in [0,1]$.

We say that real function $u (t)$, $t \in (0,1)$, is admissible shift if $u (t)$ is differentiable and
\begin{equation*}
\int_0^1 \Bigl(\frac{du(t)}{dt}\Bigr)^2\, dt < \infty, \quad u(0) = u(1) = 0.
\end{equation*}
\begin{proposition} \label{tdf21} Let $u$  be admissible shift and $\int_0^1 (du(t)/dt)^2\, dt > 0$.
Then, for any $c > 0$, we have
\begin{equation} \label{dt1}
\mathbf{P}(\,\max_{t \in [0,1]}\,|b(t)| < c) > \mathbf{P}(\,\max_{t \in [0,1]}\,|b(t)\, + \, u(t)| < c)
\end{equation}
\end{proposition}
\begin{proposition} \label{pro2} For any $c>0$ we have
\begin{equation}\label{fpro2}
\mathbf{P}(\max_{t \in [0,1]} |b(t)| < c) > \sup_{u \in \Upsilon(e_1,e_2,a)} \mathbf{P}(\max_{t \in [0,1]} |b(t) +u(t)| < c).
\end{equation}
\end{proposition}
By compactness  reasoning, Proposition \ref{pro2} follows from Proposition \ref{tdf21}. 
\begin{proof}[Proof of Proposition \ref{tdf21}]Let $\xib$ be Gaussian random vector in $\mathbb{R}^d$, $\mathbf{E} [\xib] = \0b$, $E[\xib \,\xib^T] = R$, where $R$ is positive definite matrix. Let $U$ be center-symmetric convex set in $\mathbb{R}^d$. Let $\ub \in \mathbb{R}^d$. By Andersen Theorem, we have
\begin{equation}\label{etdfq11}
 \mathbf{P}\, (\xib \in \,\ub\, + \, U\,) \le \mathbf{P}\, (\xib \in \, U\,).
\end{equation}
By Remark 2 in \cite{zal}, the equality in (\ref{etdfq11}) is attained only if set $U$ is cylindrical in direction $\ub$, that is, if $\xb \in U$ then $\xb\,+ \lambda \ub \in U$ for all $\lambda \in \mathbb{R}^1$. Now we extend reasoning of Remark 2 in \cite{zal} on setup of Proposition \ref{tdf21}.

Further reasoning are akin to Remark 2 in \cite{zal}.

For $\lambda \in \mathbb{R}^1$ denote
\begin{equation*}
I(\lambda) = \mathbf{P}(\,\max_{t \in [0,1]}\,|b(t)\, + \lambda \, u(t)| < c).
\end{equation*}
Denote
\begin{equation*}
\sigma^2 = \int_0^1 \Bigl(\frac{d u (t)}{dt} \Bigr)^2 \,d\,t
\end{equation*}
Similarly to \cite{zal}, by Cameron-Martin Theorem, for proof of Proposition \ref{tdf21}, it suffices to show that, in some vicinity of $\lambda=0$, we have
\begin{equation}\label{eqd}
 \frac{d^2 I(\lambda)}{d \lambda^2} = \mathbf{E} \Bigl[\Bigl(\Bigl(\int_0^1 \frac{d u (t)}{dt} \, d\,b(t)\Bigr)^2 - \sigma^2\Bigr) \mathbf{1}_{\{\max_{t \in\, [a,b]}\,|b(t)|\,\, <\, c\}} \Bigr] < 0.
\end{equation}
Denote $\mu(c,x)$ conditional probability of $\max_{t \in\, [a,b]}\,|b(t)| < c$ given $\int_0^1 \frac{d u (t)}{dt} \,\, d\,b(t) = x$.

Then we have
\begin{equation}\label{eqd1}
\frac{d^2 I(\lambda)}{d \lambda^2} = \frac{1}{\sqrt{2\pi}\sigma}\int_{-\infty}^\infty (x^2 - \sigma^2)\, \exp \{- x^2/(2\sigma^2)\}\, \mu(c,x)\,  dx
\end{equation}
With each value of $ d \in [\sigma,\infty)$ we associate value $e \in [0,\sigma]$ such that
\begin{equation}\label{eqd2}
\int_\sigma^d (x^2 - \sigma^2) \exp\{-x^2/(2\sigma^2)\} dx = - \int_e^\sigma (x^2 - \sigma^2) \exp\{-x^2/(2\sigma^2)\} dx .
\end{equation}
Applying Ehrhard inequality we  get  $\mu(c,d) \le \mu(c,e)$. Hence, using
\begin{equation}\label{eqd3}
\int_{-\infty}^\infty (x^2 -1) \exp\{-x^2/2\} dx =0,
\end{equation}
we get (\ref{dt1})
\end{proof}
\subsection{Proof of Theorems on uniform consistency of sets of alternatives defined in terms of densities }
Denote $\psi(x) = \int_0^x \phi  (t) d\,t$, $0 \le x \le 1$.
\begin{lemma}\label{l1} There is no sequence of functions
$$
f_n =  \sum_{j=2}^\infty \sum_{i=1}^{2^j} \theta_{nji} \phi_{ji}, \quad \|f_n \| < C,
$$
such that there holds
\begin{equation}\label{l11}
\lim_{n \to \infty} \sup_{0 \le x \le 1} \left|  \psi  (x) - F_n(x) \right| = 0,
\end{equation}
where $F_n(x) = \int_0^x f_n(t)\,d\, t$.
\end{lemma}
\begin{proof} Denote  $\Pi$  linear subspace of functions
\begin{equation*}
f =\sum_{j=2}^{\infty} \sum_{i=1}^{2^j} \theta_{ji} \phi_{ji}.
\end{equation*}
in $\mathbb{L}_2(0,1)$. Subspace $\Pi$  is closed subspace of $\mathbb{L}_2(0,1)$.

Define bounded operator
\begin{equation*}
A\,f(x) = \, \int_0^x f(s)\,d\,s, \quad f \in \mathbb{L}_2(0,1).
\end{equation*}
Note that (\ref{l11}) implies that there is sequence $f_n \in \Pi$ such that
\begin{equation}\label{lw1}
\int_0^1 \left(\psi  (x) - \int_0^x f_n(t)\,d\, t\right)^2\,d\,x\, =\, o(1)
\end{equation}
as $n \to \infty$.

Thus we will prove that (\ref{lw1}) does not hold. Suppose otherwise.

Note that integration operator $A$ has trigonometric eigenfunctions and eigenvalues $\lambda_j = c j^{-1}(1 + o(1))$ as $j \to \infty$. Therefore, expanding functions $\phi$ and $f_n$ into trigonometric series
\begin{equation*}
\phi(t) = \sum_{j=1}^\infty \theta_j \tau_j, \quad f_n = \sum_{j=1}^\infty \eta_{nj} \tau_j
\end{equation*}
with trigonometric system of functions $\tau_j$, $1 \le j < \infty$,
we can replace (\ref{lw1}) with
\begin{equation}\label{lw2}
\sum_{j=1}^\infty j^{-2}(\theta_j - \eta_{nj})^2 = o(1)
\end{equation}
Orthogonality $\phi$ and $f_n$ implies $\sum_{j}^\infty \theta_j \eta_{nj} =0$.

Denote $ r = \sum_{j=1}^\infty \theta_j^2$.

Since function $\phi$ is twice continuously differentiable then $|\theta_j| < C j^{-3/2}$ for all $1 \le j < \infty$.

   If (\ref{lw2}) holds, there are $C_n \to \infty$ and $\varepsilon_n \to 0$ as $n \to \infty$ such that
\begin{equation*}
\sum_{j=1}^{C_n}(\theta_j - \eta_{nj})^2 < \varepsilon_n.
\end{equation*}
Therefore
\begin{equation*}
\sum_{j=1}^{C_n}\theta_j \eta_{nj} = r + o(1).
\end{equation*}
Hence, using orthogonality $\phi$ and $f_n$, we have
\begin{equation}\label{lw12}
\sum_{C_n+1}^\infty \theta_j \eta_{nj} = - r + o(1).
\end{equation}
Since $|\theta_j| < C j^{-3/2}$ for all $1 \le j < \infty$, then (\ref{lw12}) implies existence of sequence  $j_n \to \infty$ as $n \to \infty$ such that $|\eta_{nj_n}| > c j_n^{3/2}$. This implies $j_n^{-2}(\theta_{j_n} - \eta_{nj_n})^2 \to \infty$ as $n \to \infty$ and we come to contradiction.\end{proof}
Denote $m_n = [\log_2 n]$.
\begin{proof}[Proof of Theorem \ref{tp1}]
 By G and Lemma \ref{l1}, it suffices to prove Theorem \ref{tp1} for sequence of functions
\begin{equation*}
f_n =\sum_{j=k_n - C_1}^{\infty} \sum_{i=1}^{2^j} \theta_{n_kji} \phi_{ji}.
\end{equation*}
Note that $\|\psi_{j1}\|_\infty = c 2^{-j/2}$ and $\|\phi_{j1}\|_\infty = c 2^{j/2}$, $1 \le j < \infty$.

Since $\|f_n\|_\infty < C_0 n^{-r}$, implementing Lemma \ref{l1}, we get that, for any $\delta > 0$, there is    $C_1$,  such that functions $\tilde F_{n\delta} = \sum_{j=k_n +C_1}^{\infty} \sum_{i=1}^{2^j} \theta_{n_kji} \psi_{ji}$ satisfy the following
\begin{equation}\label{l23}
\begin{split}&
n^{1/2}\|\, \tilde F_{n\delta} \,\|_\infty\, <\, C_2 n^{1/2} 2^{-k_n - C_1}  \|f_n\|_\infty \\& \le C_3 n^{1/2 - r} 2^{-k_n - C_1} \le C_3 2^{-C_1} \le \delta.
\end{split}
\end{equation}
Therefore, by Theorem \ref{tdf1},
there are $C$ and  $C_1$ such that,  for functions $F_{n\delta} = \sum_{j=k_n -C_1}^{k_n + C} \sum_{i=1}^{2^j} \theta_{nji} \psi_{ji}$, there holds
\begin{equation}\label{l24}
\|\,  F_{n\delta} \,\|_\infty \,\asymp\, n^{-1/2}.
\end{equation}
For any point $x \in [0,1]$  there is only finite number of indices $j$, $ k_n -c \le j \le k_n +C$ and $l$, $1 \le l \le 2^j$, such that $\theta_{njl} \phi_{jl}(x) \ne 0$. Hence by Lemma \ref{l1} and $\|\psi_{j1}\|_\infty \asymp 2^{-j/2}$, we get first statement Theorem \ref{tp1}.

Suppose $e_1 < l_n2^{-j_n} < e_2$ does not hold. Then there is subsequence $n_t \to \infty$ as $t \to \infty$ such that $l_{n_t} 2^{-j_{n_t}} \to 0$ or $l_{n_t} 2^{-j_{n_t}} \to 1$ as $t \to \infty$.
In this case, for subsequence $n_t$ Massey \cite{massey} example works, and we get that there is type I error probability  $\alpha$, $0 \le \alpha \le 1$, such that, for subsequence of Kolmogorov tests $K_{n_t}$, $\alpha(K_{n_t}) = \alpha + o(1)$, there holds
\begin{equation*}
\lim_{t \to \infty} \beta(K_{n_t}, f_{n_t}) = 1 - \alpha.
\end{equation*}
\end{proof}
\begin{proof}[ Proof of Theorem \ref{tp2}] We   begin with proof of {\it i.} in definition of
maxisets.

Let $j > k_n +C_1$. Then we get
\begin{equation}\label{l24a}
|\theta_{nji}| \le P_0 2^{-sj-j/2} \le C2^{-sk_n- j/2 -sC_1}= C\,2^{-\frac{2r}{1 - 2r} k_n}2^{-j/2-C_1s} = C\,n^{-r} 2^{-j/2-C_1s}.
\end{equation}
Therefore for any $\delta > 0$ there is $C_1$ such that $\|f_{n\delta}\|_\infty <\delta n^{-r}$ where
\begin{equation*}
f_{n\delta} = \sum_{j=k_n + C_1}^{\infty} \sum_{i=1}^{2^j} \theta_{n_kji} \phi_{ji}.
\end{equation*}
Hence there is $C$ such that  $\|\bar f_{nC}\|_\infty > C_2 n^{-r}$ where
\begin{equation*}
\bar f_{nC} = \sum_{j=k_n -c}^{k_n + C} \sum_{i=1}^{2^j} \theta_{n_kji} \phi_{ji}.
\end{equation*}
By Lemma \ref{l1}, this implies that there is sequences $j_n$, $k_n -c < j_n < k_n +C$ and $i_n$, $1 \le i_n \le 2^{j_n}$ such that $|\theta_{nj_ni_n}| > C_3 n^{-r} 2^{-j_n/2}  = Cn^{-1/4 -r/2}$.  By Theorem \ref{tp1} this implies consistency of sequence $f_n$.

It remains to verify {\it ii.} in definition of maxisets. Suppose opposite.  Let, for function   $f = \sum_{j=1}^{\infty} \sum_{i=1}^{2^j} \theta_{ji} \phi_{ji}$ we can point out sequences $j_t \to \infty$  as $t \to \infty$ and $l_t$, $ 1 \le l_t \le 2^{j_t}$, such that there holds
\begin{equation}\label{va1}
2^{j_t(s + 1/2)} |\theta_{j_tl_t}| = C_t,
\end{equation}
where $C_t \to \infty$  as $t \to \infty$.

Denote $c_t = \log_2 C_t$.

We put $\eta_{n_tji} =0$ for $1 \le j \le j_t$,  $1 \le i \le 2^{j_t}$ and  $\eta_{n_tji} = \theta_{ji}$ for $j > j_t$,  $1 \le i \le 2^{j_t}$.
Define subsequence of alternatives $f_{n_t} = \sum_{j=1}^{\infty} \sum_{i=1}^{2^j} \eta_{n_tji} \phi_{ji}$.
Here $n_t$ satisfies $\|\, f_{n_t}\,\|_\infty  \asymp n_t^{-r}$.

Then, by Lemma \ref{l1}, we get
\begin{equation}\label{va2}
\|\, f_{n_t}\,\|_\infty  \asymp 2^{j_t/2} |\theta_{j_tl_t}| \asymp 2^{-m_tr} \asymp n_t^{-r},
\end{equation}
where $m_t = [\log_2 n_t]$.

By (\ref{va1}) and (\ref{va2}), we  get
\begin{equation}\label{va3}
-m_t\,r = j_t/2 - j_t\,(s + 1/2)r + c_t
\end{equation}
To verify { \it ii.}, by inequality of Dvoretzky, Kiefer, Wolfowitz \cite{dv}, it suffices to show that there holds
\begin{equation}\label{va4}
2^{m_t/2} 2^{-j_t/2} \theta_{j_tl_t} = o(1)
\end{equation}
By (\ref{va1}), we get
\begin{equation}\label{va5}
2^{m_t/2} 2^{-j_t/2} \theta_{j_tl_t} \asymp 2^{m_t/2} C_t  2^{-j_t(s+1)}
\end{equation}
Hence, using (\ref{va3}) and $s - s/(2r) +1 = 0$, we get
\begin{equation}\label{va6}
2^{m_t/2} C_t  2^{-j_t(s+1)} \asymp C_t^{1 -1/(2r)} 2^{-l_t(s - s/(2r) +1)} \asymp C_t^{1 -1/(2r)} = o(1).
\end{equation}
This implies inconsistency of sequence of alternatives $f_{n_t}$.
\end{proof}
 \begin{proof}[Proof of Theorem \ref{tp2a}] First statement of Theorem \ref{tp2a} is easily deduced from Lemma \ref{s1} given below.
 \begin{lemma}\label{s1} Let $f_n = \sum_{j=1}^{k_n+c} \sum_{i=1}^{2^j} \theta_{ji} \phi_{ji}$, $\|f_n\|_\infty \asymp n^{-r}$. There is $P_0$ such that $f_n \in \mathbb{B}^s_{\infty\infty}(P_0)$.
\end{lemma}
\begin{proof} Using $\|f_n\|_\infty \asymp n^{-r}$ and implementing \ref{l1} we get
\begin{equation*}
|\theta_{nji}| \le C n^{-r}2^{-j/2} = C 2^{-\frac{2r}{1 - 2r} k_n} 2 ^{-j/2} = C 2^{-s\,k_n} 2^{-j/2} \le 2^{-sj -j/2}.
\end{equation*}
\end{proof}
  For proof of first statement of Theorem \ref{tp2a} it suffices to put
$$
f_{1n} =  \sum_{j=1}^{k_n +C} \sum_{i=1}^{2^j} \theta_{nji} \phi_{ji},
$$
with constant $C$ defined in Theorem \ref{tp1}.  By Lemma \ref{s1}, we get $f_{1n} \in \mathbb{B}^s_{\infty\infty}(P_0)$.     By Theorem \ref{tp1},  there are $j_n$ and $i_n$ with $k_n - c < j_n < k_n +C$ and $1 \le i_n \le 2^{j_n}$, such that $|\theta_{nj_ni_n}| \asymp n^{-r} 2^{-j_n/2}$. Therefore, by Lemma  \ref{l1},  $\|f_{1n}\|_\infty \asymp n^{-r}$. By Lemma \ref{s1} this implies first statement of Theorem \ref{tp2a}.

Second statement of Theorem \ref{tp2a} is easily deduced from proof of {\sl i.} in Theorem \ref{tp2}.
\end{proof}

\begin{lemma} \label{l0} Let  $h_1(t)$ and $h_2(t)$, $t \in (a,b)$, $0< a < b < 1$ be two bounded real functions. Suppose
\begin{equation}\label{l01}
\max_{a < t < b} |h_1(t) - h_2(t)| < \delta, \quad \delta > 0.
\end{equation}
Then, for any $c$, there holds
\begin{equation}\label{l02}
\begin{split}&
|\mathbf{P}\,(\,\max_{a < t < b} |b(t) + h_1(t)| <\, c)  -
\mathbf{P}\,(\,\max_{a < t < b} |b(t) + h_2(t)| <\, c)|
\le C \,\delta,
\end{split}
\end{equation}
where $C$ does not depend on $h_1$ e $h_2$.
\end{lemma}
\begin{proof} By (\ref{l01})we get  $h_1(t) - \delta < h_2(t) < h_1(t) + \delta$ for $a < t < b$. Therefore, for proof of (\ref{l02}) it suffices to show
\begin{equation}\label{l03}
\begin{split}&
\mathbf{P}\,(\,\max_{a < t < b} |b(t) + h_1(t)| <\, c+ \delta)  - \mathbf{P}\,(\,\max_{a < t < b} |b(t) + h_1(t)| <\, c)   \\&
\le \Phi\Bigl(\frac{\delta}{a}\Bigr) - \Phi\Bigl(-\frac{\delta}{a}\Bigr) +  \Phi\Bigl(\frac{\delta}{1-b}\Bigr) -  \Phi\Bigl(-\frac{\delta}{1-b}\Bigr).
\end{split}
\end{equation}
Define function $h(t) = \sqrt{d- d^2} t/d $ for $0 < t < d$ and $h(t) = \sqrt{d- d^2}  \frac{1 -t}{1-d}$ for  $d \le t <1$, where $a < d < b$.

Then Brownian bridge can be written in the following form: $b(t) = \zeta(t) + \xi h(t)$, where Gaussian random process $\zeta(t)$  does not depend on random variable $\xi$, $\mathbf{E} [\xi] = 0$ and $\mathbf{E} [\xi^2] =1$.

Denote $\mu_\eta$  probability measure of random process $\eta(t) = \zeta(t) + h_1(t)$. Then right hand-side of (\ref{l02}) equals
\begin{equation}\label{l04}
\int\, d\,\mu_\eta (\mathbf{P}\,(\,\max_{a < t < b} |\eta(t) + \xi h(t) + h_1(t)| <\, c+ \delta)  - \mathbf{P}\,(\,\max_{a < t < b} |\eta(t) + \xi h(t) + h_1(t)| <\, c)).
\end{equation}
Thus it suffices to prove that for any bounded real function a $S\, : \, (a,b) \to \mathbb{R}^1$  there holds
\begin{equation}\label{l05}
\mathbf{P}\,(\,\max_{a < t < b} |S(t) + \xi h(t)| <\, c+ \delta)  - \mathbf{P}\,(\,\max_{a < t < b} |S(t) + h(t)| < \, c) \le C\,\delta,
\end{equation}
where $C$ does not depends on $S$.

 We prove slightly more simple statement for function  $h(t) = t$, $ t \in (0,1)$. In our case reasoning are more cumbersome.

 We can write  (\ref{l05}) in the following form
\begin{equation}\label{l06}
\begin{split}&
\mathbf{P}\,\Bigl(\,-\min_{a < t < b}\Bigl\{\frac{c+\delta}{t} + \frac{S(t)}{t}\Bigr\}   < \, \xi \,  < \,\min_{a < t < b}\Bigl\{\frac{c+\delta}{t} + \frac{S(t)}{t}\Bigr\}\Bigr)\\& - \mathbf{P}\,\Bigl(\,-\min_{a < t < b}\Bigl\{\frac{c}{t} + \frac{S(t)}{t}\Bigr\}   < \, \xi \,  < \,\min_{a < t < b}\Bigl\{\frac{c}{t} + \frac{S(t)}{t}\Bigr\}\Bigr)  \\& \le
\Phi\left(\frac{\delta}{a}\right) - \Phi\left(-\frac{\delta}{a}\right).
\end{split}
\end{equation}
Hence, using (\ref{l04}), (\ref{l06}), we get          (\ref{l02}).
\end{proof}

\begin{proof}[Proof of Theorem \ref{tp3}] Reasoning are based on Lemma  \ref{l0}.
Denote
\begin{equation*}
\bar f_{nC} = \sum_{j=1}^{k_n+C} \sum_{i=1}^{2^j} \theta_{ji} \phi_{ji}
\end{equation*}
and $\tilde f_{nC} = f_n - \bar f_{nC}$.

Define function $\bar F_{nC} $ such that $\frac{d \bar F_{nC}(x)}{d\,x} = 1 + \bar f_{nC}(x)$, $x \in [0,1]$ and $\bar F_{nC}(1) = 1$. Denote $\tilde F_{nC} = F_n - \bar F_{nC} + F_0$.

By Lemma \ref{s1}, for any $C$, there is  $P_0$ such that $\bar f_{nC} \in \mathbb{B}^s_{\infty\infty}(P_0)$. By Theorem \ref{tp1}  and Lemma \ref{l1}, there is  $C$  such that  $\| \bar f_{nC}\| \asymp n^{-r}$. This implies {\it i.}

By (\ref{l23}), for any $\delta> 0$ we can choose $C$ such that $n^{1/2} T(\tilde F_{nC}) < \delta$. By Lemma \ref{l0}, this implies {\it i.} and {\it ii.}
\end{proof}
If sequence $F_{1n}$ is inconsistent, then, by Theorem \ref{tdf4}, $n^{1/2} T(F_{1n})  \to 0$ as $n \to \infty$. Hence, by Lemma \ref{l0}, we get Theorem \ref{tp4}.

Theorem \ref{tq6} are easily deduced from Theorem \ref{tp101} and proof is omitted.

Theorem \ref{tq12} follows from Lemma \ref{s1}  and Theorem \ref{tq6}.
\subsection{Proof of Theorem \ref{tco2}}
We  prove only second statement of Theorem \ref{tco2}.

Let $m_n \to \infty$ be such a sequence that there is
\begin{equation*}
|\theta_{m_ni_{m_n}}| > (1 -\delta)\max\{|\theta_{ji}|\,:\, j \ge m_n, \, 1 \le i \le 2^j\}
 \end{equation*}
 with  $0 < \delta < 1/2$.

Define subsequence $l_{m_n}^{-1/2} \asymp 2^{m_n} \, |\theta_{ m_ni_{m_n}}|$.

Define subsequence of functions $f_{l_{m_n}} = \sum_{j=l_{m_n}}^\infty \sum_{i=1}^{2^j} \theta_{nji} \phi_{ji}$.

Denote $F_{l_{m_n}}(x) = x + \int_0^x f_{l_{m_n}}(t)\,d\, t$, $x \in [0,1]$.

Then, by Lemma \ref{l1}, there holds
\begin{equation*}
\max_{x \in (0.1)} |\, F_{l_{m_n}}(x) - F_0(x)\,| \asymp 2^{-m_n} l_{m_n}^{-1/2} = o(l_{m_n}^{-1/2}).
\end{equation*}
By Theorem \ref{tdf4}, this implies inconsistency of subsequence of alternatives  $F_{l_{m_n}}$.


\begin{thebibliography}{99}

\bibitem{chib} D.M. Chibisov,    "An investigation of the asymptotic power of tests of fit," {\it Theor.Prob. Appl.,} {\bf 10}, 421 -437 (1965).

    \bibitem{durb}  J. Durbin,  {\it Distribution Theory for Tests Based on the Sample Distribution function}. Regional Conference Series in Applied Mathematics, v.9, SIAM,   Philadelphia (1973).


\bibitem{dv}  A. Dvoretzky, J. Kiefer and J. Wolfowitz,   "Asymptotic minimax character of the sample distribution function and of the classical multinomial estimator," {\it Ann. Math. Statist.,} {\bf 27}, 642 - 669 (1956).

    \bibitem{er97} M.S.  Ermakov,   "Asymptotic minimaxity of chi-squared tests," {\it Theory Probab. Appl.,} {\bf 42}, 589 - 610 (1997).

\bibitem{er19}  M.S. Ermakov,    "On uniform consistency  of nonparametric tests I," {\it Zapiski Nauchnih  Seminarov POMI RAS.} {\bf 486}, 98-147 (in Russian) arXiv:1807.09076 (2019).

\bibitem{gine} E. Gine  and R. Nickl, {\it Mathematical Foundation of Infinite--Dimensional Statistical Models}. Cambridge University Press, Cambridge (2015).

\bibitem{ha} J. Hajek, Z. Sidak, and P.K. Sen, {\it Theory of Rank Tests}. Academic Press, London (1999).

    \bibitem{ing87} Yu.I. Ingster,   "On comparison of the minimax properties of Kolmogorov, $\omega^2$ and $\chi^2$-tests," {\it Theory. Probab. Appl.,} {\bf 32}, 346-350 (1987).

 \bibitem{jo} I. M.  Johnstone,   {\it Gaussian estimation. Sequence and wavelet models}. Book Draft http://statweb.stanford.edu/~imj/ (2015).

     \bibitem{ken} M. Kendall  and  A. Styuart, {\it The Advanced Theory of Statistics} v. 2, Charles Griffin,  London (1960).

 \bibitem{kol} A.Kolmogorov,  "Sulla determinazione empirica di una legge di distribuzione," {\it G. Ist. Ital. Attuari.,} { \bf 4}, 83 - 91 (1933).

\bibitem{le} E.L. Lehmann and J.P. Romano,  {\it Testing Statistical Hypothesis}. Springer Verlag, NY. (2005).

    \bibitem{lep}  O.V.   Lepski and A.B. Tsybakov,  "Asymptotically exact nonparametric hypothesis
testing in sup-norm and at a fixed point," {\it Probab. Th. Rel.
Fields}, {\bf 117}:1, 17 - 48 (2000).

\bibitem{lif}  M. Lifshits,  On the absolute continuity of the distributions of functionals of stochastic processes”, {\it Th. Probab. Appl.}, {\bf 27}:3  600–607 (1983).

\bibitem{mas}  P. Massart,   "The tight constant in the Dvoretzky, Kiefer, Wolfowitz inequality", {\it Ann. Probab.,} {\bf 18}:3, 1269 - 1283 (1990).

\bibitem{mason}  D. Mason  and J. Schuenemeyer,   "A Modified Kolmogorov-Smirnov Test Sensitive to Tail Alternatives," {\it  Ann. Statist.}, 11(3), 933 - 946 (1983).

 \bibitem{massey}  F. J. Massey,  "A note on the power of a non-parametric test," {\it Ann.
Math. Statist,} {\bf 21}, 440 - 443 (1950).

    \bibitem{wel} G.R.Shorack and J.A. Wellner,   {\it Empirical Processes with Application to Statistics}, J.Wiley Sons,  NY  (1986).

      \bibitem{sm1}   N. V.Smirnov,   "On deviations of empirical distribution function," {\it Mathematical Sbornik}, {\bf 6}:1, 3 - 26 (In Russian) (1939).

      \bibitem{sm2}  N. V.Smirnov,    "Approximate laws of distribution of random variables from empirical data," {\it Usp. Mat. Nauk.,} {\bf 10}, 179 - 206 (1944).

               \bibitem{zal} V.A. Zalgaller,  "Mixed volumes and probability of hitting in convex domain for a multidimensional normal distribution," {\it Mathematical Notes of the Academy of Sciences of the USSR,} {\bf 2}, 542-545 (1967).
\end{thebibliography}
\end{document}